\documentclass{amsart}

\usepackage{amsmath,mathtools,amssymb,amsthm,comment,color,soul}

\newtheorem{thm}{Theorem}
\newtheorem{prop}[thm]{Proposition}
\newtheorem{lem}[thm]{Lemma}
\newtheorem{cor}[thm]{Corollary}
\newtheorem*{open}{Question}

\title[Immersed M\"{o}bius bands and $\mathbb{Z}_2$-homology classes]{Immersed M\"{o}bius bands in knot complements and representatives of $\mathbb{Z}_2$-homology classes}
\author{Mark C. Hughes and Seungwon Kim}

\begin{document}

\begin{abstract}
We study the 3-dimensional immersed crosscap number of a knot, which is a nonorientable analogue of the immersed Seifert genus.  We study knots with immersed crosscap number 1, and show that a knot has immersed crosscap number 1 if and only if it is a nonntrivial $(2p,q)$-torus or $(2p,q)$-cable knot.  We show that unlike in the orientable case the immersed crosscap number can differ from the embedded crosscap number by arbitrarily large amounts, and that it is neither bounded below nor above by the 4-dimensional crosscap number.  We then use these constructions to find, for any $n\geq 2$, an oriented 3-manifold $Y_n$ and class $\alpha_n \in H_2(Y_n;\mathbb{Z}_2)$ such that $\alpha_n$ can be represented by an immersed $\mathbb{RP}^2$, but any embedded representative of $\alpha_n$ has a component $S$ with $\chi(S) \leq 1-n$.  
\end{abstract}

\maketitle

\section{Introduction}

One method to define a measure of the complexity of a knot $K \subset S^3$ is by describing the minimal topological complexity of a compact surface $F$ whose boundary is equal to $K$.  When $F$ is required to be orientable and embedded in $S^3$ this gives the classical \emph{Seifert genus} $g_3(K)$ of $K$; when the requirement that $F$ be embedded in $S^3$ is weakened, and we allow embedded surfaces in $B^4$ we obtain the \emph{slice genus} $g_4(K)$ of $K$.  Here we are thinking of $S^3$ as the boundary of $B^4$.

Loosening the requirement that $F$ be embedded in $S^3$ in another direction we can instead consider surfaces $F$ which are merely immersed in $S^3$.  While it is easy to verify that every knot $K$ is the boundary of an immersed disk in $S^3$, if we consider only immersed surfaces which are embedded along a neighborhood of their boundaries we obtain \emph{immersed Seifert surfaces} for $K$.  The minimal genus of any immersed Seifert surface for $K$ is called the \emph{immersed Seifert genus} of $K$, and is denoted $g_I(K)$.  While $g_I$ takes nontrivial values on knots in $S^3$, using foliations Gabai \cite{gabai1983foliations} proved that the resulting knot invariant is always equal to the Seifert genus.  Both invariants are in turn bounded below by the slice genus.

Dropping the requirement that $F$ be orientable, we can instead consider nonorientable immersed spanning surfaces, defined in a similar way as above.  This gives rise to a nonorientable analogue of the immersed Seifert genus, called the \emph{immersed crosscap number} $\gamma_I(K)$ of the knot $K$.  Our main result involves knots with immersed crosscap number 1, i.e.,\ knots which bound immersed M\"{o}bius bands that are embedded along their boundaries.

For $|p| \geq 2$, we say that $K$ is a \emph{$(p,q)$-cable knot} if it can be isotoped to lie on the boundary of a solid torus $V \subset S^3$ whose core is knotted, where $K$ represents the class $q\cdot m + p \cdot l$ in $\pi_1(\partial V) \cong \mathbb{Z} \oplus \mathbb{Z}$.  Here $m$ and $l$ are the homotopy classes of the meridian and the Seifert-framed longitude of $\partial V$ respectively.

\begin{thm}
\label{thm:cableknots}
A knot $K \subset S^3$ has $\gamma_I (K) =1$ if and only if $K$ is a nontrivial $(2p,q)$-torus or $(2p,q)$-cable knot.
\end{thm}

In a similar way, we can also define the \emph{3-dimensional (embedded) crosscap number} $\gamma_3(K)$ and \emph{4-dimensional (embedded) crosscap number} $\gamma_4(K)$ of a knot, which are nonorientable analogues of the Seifert and slice genus respectively.  Theorem~\ref{thm:cableknots} then generalizes a result of Clark \cite{clark1978crosscaps}, who proved that $\gamma_3 (K) = 1$ if and only if $K$ is a $(2,q)$-torus or $(2,q)$-cable knot.

Unlike their orientable counterparts, in general $\gamma_I (K)$ may not equal $\gamma_3(K)$, and is not bounded below by $\gamma_4(K)$.  More precisely, we present an infinite family of immersed crosscap number 1 knots with unbounded 3- and 4-dimensional crosscap numbers.  Furthermore, we also present examples of knots $K$ with $\gamma_I(K) > \gamma_4(K)$.


We also highlight the difference between the topological complexity of immersed and embedded nonorientable surfaces by studying immersed and embedded representatives of certain homology classes in 3-manifolds.  It is well-known that when $Y$ is an oriented 3-manifold, any class in $H_2 (Y;\mathbb{Z})$ can be represented by a closed embedded orientable surface.  Gabai's work implies that the minimum topological complexity for a representative of a given $\alpha \in H_2(Y;\mathbb{Z})$ (as measured by the Thurston norm) is the same whether the surface representing $\alpha$ is required to be embedded or merely immersed.

It can similarly be shown that any class in $H_2 (Y;\mathbb{Z}_2)$ can also be represented by a closed embedded surface.  In contrast to the orientable case, however, the minimal genus of immersed and embedded representatives of a given class can differ by arbitrarily large amounts.  This fact is illustrated by the following theorem, which is proven in Section~\ref{sec:homology}.
\begin{thm}
\label{thm:homology}
For every integer $n \geq 2$ there exists a closed oriented 3-manifold $Y_n$ and homology class $\alpha_n \in H_2 (Y_n ; \mathbb{Z}_2)$ such that 
\begin{enumerate}
\item $\alpha_n$ can be represented by an immersed $\mathbb{RP}^2$, and 
\item any embedded representative of $\alpha_n$ has a component $S$ with $\chi(S) \leq 1-n$.
\end{enumerate}
\end{thm}

\section{The immersed crosscap number of a knot}
\label{sec:immersedcrosscap}

\subsection{Crosscap numbers of knots} We begin by defining nonorientable analogues of the Seifert, slice, and immersed Seifert genera of knots.  Roughly speaking, these values capture the minimum $k$ needed to span the knot by a punctured connected sum of $k$ copies of $\mathbb{RP}^2$, assuming different embedding and immersion requirements. 

Let $K$ be a knot in $S^3$.  We say that a compact, embedded, nonorientable surface $F \subset S^3$  with $\partial F = K$ is a \emph{nonorientable spanning surface} for $K$, and we define the \emph{3-dimensional (embedded) crosscap number} of a nontrivial knot $K$ to be 
\[
\gamma_3(K) = \min \left\{ b^1(F) \, | \, F \text{ is a nonorientable spanning surface for } K \right\}.
\] 
Here $b^1(F)$ is the first Betti number of $F$.  We define $\gamma_3$ of the unknot to be $0$.

If we think of $S^3 = \partial B^4$, then a compact, embedded, nonorientable surface $F \subset B^4$  with $\partial F = K \subset S^3$ is a \emph{nonorientable slice surface} for $K$, and we define the \emph{4-dimensional (embedded) crosscap number} of a nonslice knot $K$ (i.e. $g_4(K) \geq 1$) to be 
\[
\gamma_4(K) = \min \left\{ b^1(F) \, | \, F \text{ is a nonorientable slice surface for } K \right\}.
\] 
If $K$ bounds an embedded slice disk in $B^4$ we defined $\gamma_4(K) = 0$.

Lastly, suppose that $F$ is the image of an immersion $h : \Sigma \rightarrow S^3$, where $\Sigma$ is a compact, nonorientable surface with boundary.  Then $F = h(\Sigma)$ is a \emph{nonorientable immersed spanning surface} for $K$ if $h(\partial \Sigma)=K$, and if there is a collar neighborhood $A$ of the boundary $\partial \Sigma$, such that $h(A)$ is embedded, and $h^{-1}(h(A))=A$.  The first Betti number $b^1(F)$ of the nonorientable immersed spanning surface $F$ is defined to be $b^1(\Sigma)$, and if $K$ is a nontrivial knot we define the \emph{nonorientable immersed crosscap number} of $K$ to be 
\[
\gamma_I(K) = \min \left\{ b^1(F) \, | \, F \text{ is a nonorientable immersed spanning surface for } K \right\}.
\]
In the case when $K$ is the unknot, we again define $\gamma_I(K) = 0$.

Recall that in the orientable case, the Seifert, slice, and immersed Seifert genus of a knot $K$ satisfy
\[
g_I(K) = g_3(K) \geq g_4(K).
\]
Our goal in this section is to determine which of the following results above generalize to the nonorientable case.  

Firstly, as any nonorientable spanning surface in $S^3$ can be pushed into $B^4$ to become a nonorientable slice surface, we clearly have that $\gamma_3 (K) \geq \gamma_4(K)$.  Furthermore, we also trivially have  that $\gamma_3 (K) \geq \gamma_I (K)$.

Note, however, that not every nonorientable immersed spanning surface can be pushed to an embedding in $B^4$ (see \cite{carter1998surfaces} for criteria describing when this is possible).  Furthermore, not every nonorientable slice surface can be pushed into $S^3$ to give a nonorientable immersed spanning surface.  Hence, we do not have any a priori relations between the invariants $\gamma_I(K)$ and $\gamma_4(K)$, a fact which can be illustrated with a few simple examples.

In \cite{teragaito2004crosscap} Teragaito gives and algorithm for computing the crosscap number of the $(p,q)$-torus knot $T(p,q)$ using the partial fraction expansions of rational expressions involving $p$ and $q$.  In the case of $T(2k,2k-1)$, where $k\geq 2$, his results give $\gamma_3 (T(2k,2k-1)) = k$.

On the other hand Batson \cite{batson2012nonorientable} finds a lower bound on $\gamma_4(K)$ involving the signature of $K$, along with the Heegaard-Floer $d$-invariant of certain integer homology spheres.  In the case of $T(2k,2k-1)$, again with $k\geq2$, his results specialize to give $\gamma_4 (T(2k,2k-1)) = k-1$.

\begin{prop}
\label{prop:crosscap1}
Let $p$ and $q$ be integers, with $2p$ and $q$ coprime, and $|q| \geq 2$.  Then $\gamma_I (T(2p,q)) = 1$.
\end{prop}

\begin{proof}
With $p$ and $q$ as above, the knot $K=T(2p,q)$ is nontrivial, and hence $\gamma_I (K) \geq 1$.  Furthermore, it is not difficult to construct an immersed M\"{o}bius band $F$ whose boundary is $K$ as follows.  

Suppose that $K$ is embedded along the boundary of the standard embedding of the solid torus $V \subset S^3$.  Suppose that for any disk $D_\theta \subset V$ of the form $\left\{ \theta = \text{constant} \right\} \cap V$, where $\theta$ is the angular polar coordinate, $K \cap D_\theta$ is a collection of $2p$ evenly-spaced points $\left\{x_1, \ldots, x_{2p} \right\}$ around $\partial D_\theta$.  In each $D_\theta$, draw $p$ straight lines through the center of $D_\theta$, connecting $x_j$ with $x_{p+j}$ for $1\leq j \leq p$.  As $\theta$ ranges from $0$ to $2\pi$ these lines will sweep out an immersed M\"{o}bius band in $V$, with boundary $K$ (see Figure~\ref{fig:immersedmobius}).  Furthermore, this M\"{o}bius band will be embedded away from the core of $V$.
\end{proof}

\begin{figure}
    \centering
    \includegraphics[width=0.4\textwidth]{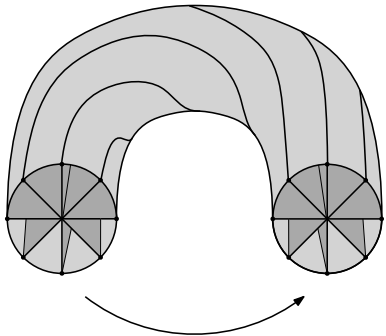}
    \caption{An immersed M\"{o}bius band in a solid torus with boundary $T(2p,q)$.}
    \label{fig:immersedmobius}
\end{figure}

Clearly the above proof of Proposition~\ref{prop:crosscap1} generalizes to $(2p,q)$-cable knots, a fact which we record here.   
\begin{prop}
If $K$ is a $(2p,q)$-cable knot, then $\gamma_I(K)=1$.
\end{prop}

In particular, for $k\geq 2$, we have $\gamma_I (T(2k,2k-1)) = 1$.  We thus see that both of the quantities $\gamma_3(K)-\gamma_I(K)$ and $\gamma_4(K)-\gamma_I(K)$ can be arbitrarily large.  On the other hand, we can also have $\gamma_I (K)> \gamma_4 (K)$.  Indeed, an immediate corollary to Theorem~\ref{thm:cableknots} is that any hyperbolic knot $K$ has $\gamma_I (K)>1$.  Hence, any slice hyperbolic knot $K$ has $\gamma_I (K)> \gamma_4 (K)$.  The Stevedore's knot $6_1$ is the simplest example of such a knot.  An interesting question would be to ask whether the value of $\gamma_I(K) - \gamma_4(K)$ can be arbitrarily large.

\begin{open}
For any $n \in \mathbb{N}$ does there exist a knot $K$ such that $\gamma_I(K) - \gamma_4(K) > n$?
\end{open}

\section{Essential M\"{o}bius bands}
\label{sec:essential}

In what follows we focus on studying nontrivial knots $K$ which bound immersed M\"{o}bius bands that are embedded near their boundaries, i.e.\ knots with  $\gamma_I (K) = 1$.  To do so it will be more useful to think of our immersed nonorientable spanning surfaces as lying in the exterior of $K$, rather than $S^3$ itself.

More precisely, let $N(K) \subset S^3$ be a small open tubular neighborhood of $K$, and let $E(K) = S^3 \backslash N(K)$ be the exterior of $K$.  Then $\partial E(K)$ is a torus, with a canonical choice of meridian.  An \emph{immersed (nonorientable) spanning surface} $F$ in $E(K)$ is the image of an immersion $h:\Sigma \rightarrow E(K)$, where $\Sigma$ is a (nonorientable) compact surface with boundary, such that $h(\partial \Sigma) = F \cap \partial E(K)$ is a longitude on $\partial E(K)$, and $\partial \Sigma$ has a collared neighborhood $A$ in $\Sigma$ such that $h(A)$ is embedded and $h^{-1}(h(A))=A$.  Clearly there is a straightforward way to pass between nonorientable immersed spanning surfaces for the knot $K$ in $S^3$, and nonorientable immersed spanning surfaces in the knot exterior $E(K)$.  Furthermore, a nontrivial knot $K$ has $\gamma_I (K) = 1$ if and only if there is a spanning surface in $E(K)$ which is the immersed image of a M\"{o}bius band.

\subsection{Essential maps}
\label{sec:essentialmaps}

Consider a map $h: \Sigma \rightarrow E(K)$ with $h(\partial \Sigma) \subset \partial E(K)$.  We say that $h$ is \emph{$\pi_1$-essential} if $h_*: \pi_1(\Sigma) \rightarrow \pi_1(E(K))$ is injective. Similarly, we say that $h$ is \emph{$\partial\pi_1$-essential} if $h_* : \pi_1(\Sigma, \partial \Sigma) \rightarrow \pi_1(E(K), \partial E(K))$ is injective.  Finally, we say that $h$ is \emph{essential} if it is both $\pi_1$- and $\partial \pi_1$-essential.  We will sometimes describe the image of an embedding $h:\Sigma \rightarrow E(K)$ as being \emph{essential}, if the map $h$ is essential.

Now let $K \subset S^3$ be a nontrivial knot, and $M$ a M\"{o}bius band.  To make things more precise, let $M$ be given by the square $[-1,1] \times [-1,1]$ in $\mathbb{R}^2$, with vertical edges identified via $(-1,t) \sim (1,-t)$.  The \emph{core} of $M$ will be denoted by $c$, and is the image of $[-1,1]\times \{ 0\}$ under the quotient map.  Let $\alpha$ be the image of the arc $\left\{0\right\} \times [-1,1]$ under the quotient map.  Note that the homotopy class of $c$ generates $\pi_1(M) \cong \mathbb{Z}$, and the homotopy class of $\alpha$ is the only nontrivial class in $\pi_1(M, \partial M)$.

We fix a meridian $m$ and longitude $l$ of $\partial E(K)$, which we think of as generators for $\pi_1(\partial E(K)) \cong \mathbb{Z} \oplus \mathbb{Z}$.  (We will write the group operation of $\pi_1(\partial E(K))$ as addition because it is abelian, and suppress explicit reference to a basepoint when there is no danger in doing so.)

\begin{thm}
\label{thm:essential}
Let $K$ be a nontrivial knot, and suppose that $h:M \rightarrow E(K)$ is a proper map where $h(\partial M)$ is homotopic in $\partial E(K)$ to a curve of the form $a \cdot m + b\cdot l$, for some $a,b \in \mathbb{Z}$ with either $a$ or $b$ odd.  Then $h$ is essential.  
\end{thm}

\begin{proof}
We begin by showing that $h$ is $\pi_1$-essential.  Note that since $\pi_1 (E(K))$ is torsion-free, it suffices to show that $h_* (c)$ is nontrivial.  Suppose then to the contrary that $h_*(c)$ is null-homotopic.

Consider a tubular neighborhood $N(c) \subset M$ of $c$.  Notice that $\partial N(c)$ is a double of the core, and hence $h(\partial N(c))$ is also null-homotopic in $E(K)$.  Hence we can take a disk $D$ and glue it to $M\backslash N(c)$, via some identification $\varphi$ of $\partial D$ with $\partial N(c)$, and then extend the map $h$ across $D$ to get a map $h': M \backslash N(c) \cup_\varphi D \rightarrow E(K)$.  Note, however, that $M \backslash N(c) \cup_\varphi D$ is homeomorphic to a disk $D^2$, and hence we obtain a map $h':D^2 \rightarrow E(K)$, with $h'(\partial D^2) = h(\partial M)$.  As $h(\partial M)$ is nontrivial in $\pi_1(\partial E(K))$, the Loop Theorem then implies that there is a properly embedded disk in $D' \subset E(K)$, with $\partial D'$ nontrivial in $\pi_1(\partial E(K))$.  This contradicts the assumption that $K$ was nontrivial, and hence $h$ must be $\pi_1$-essential.

To show that $h$ is $\partial \pi_1$-essential, assume now that there is homotopy taking $h(\alpha)$ to $\partial E(K)$, relative to $\partial h(\alpha)$.  Using this homotopy, we can modify $h$ to obtain a new map $h_0:M \rightarrow E(K)$, which sends $\partial M \cup N(\alpha)$ to $\partial E(K)$, where $N(\alpha)$ is a tubular neighborhood of $\alpha$ in $M$.  Moreover, the restrictions of $h$ and $h_0$ to $\partial M$ will be homotopic inside $\partial E(K)$.

Consider now the disk $D_0 = M \backslash N(\alpha)$.  The map $h_0$ restricts to give $h_0:D_0 \rightarrow E(K)$, with $h_0(\partial D_0) \subset E(K)$.  Suppose first that $h_0(\partial D_0)$ is a nontrivial loop in $\pi_1(\partial E(K))$.  Then as above the Loop Theorem implies that $K$ is the trivial knot, which is a contradiction.

Suppose then that $h_0(\partial D_0)$ is null-homotopic in $\partial E(K)$.  This null-homotopy can be viewed as a map $\rho:D_0 \rightarrow \partial E(K)$, where $\rho|_{\partial D_0} \equiv h_0|_{\partial D_0} $.  Then we can define a map $h_1 : M \rightarrow \partial E(K)$ by
\[
h_1(x) = \begin{cases}
\rho(x) \, \, \, \qquad \text{for } x \in D_0\\
h_0(x) \qquad \text{for } x \in N(\alpha)\\
\end{cases}.
\]
Furthermore, we have that $h_1|_{\partial M} \equiv h_0|_{\partial M}$, which is homotopic to $h|_{\partial M}$ in $\partial E(K)$.  

Note, however, that $\partial M$ is homotopic to $2 \cdot c$ in $M$.  Hence in $\pi_1 (\partial E(K))$ the homotopy class of $h(\partial M)=a \cdot m + b\cdot l$ will be two times the homotopy class of $h_1(c)$, which contradicts the assumption that at least one of $a$ or $b$ is odd.  Thus $h$ must be $\partial \pi_1$-essential, and therefore esssential.
\end{proof}

\section{Immersed crosscap number 1 knots}
\label{sec:cable}

For the remainder of the paper, let $K$ be a knot with $\gamma_I (K) = 1$, $M$ a M\"{o}bius band, and $h:M \rightarrow E(K)$ an immersion such that $F=h(M)$ is an immersed nonorientable spanning surface in $E(K)$.  We will make use of the following theorem of Cannon and Feustel.  Note that we can define the notions of $\pi_1$-essential, $\partial\pi_1$-essential, and essential as in Section~\ref{sec:essentialmaps} for maps into any manifold $Y$.  Let $A$ be an annulus.  
\begin{thm}[\cite{cannon1976essential}, Theorem 4]
\label{thm:cannonandfeustel}
Let $Y$ be a compact orientable 3-manifold, and $h:A \rightarrow Y$ an essential map with $h(\partial A) \subset \partial Y$.  Then there exists an essential embedding $f:\Sigma \rightarrow Y$ with $f(\partial \Sigma) \subset \partial Y$, where $\Sigma$ is either an annulus or M\"{o}bius band.  Moreover, if $h|_{\partial A}$ is an embedding, then we may assume that $f(\partial \Sigma) \subset h (\partial A)$.
\end{thm}

We begin the proof of Theorem~\ref{thm:cableknots} by first showing that $K$ must be a torus or cable knot, before describing its type. 

\begin{lem}
\label{lem:cableortorus}
If $\gamma_I(K) = 1$, then $K$ is a torus or cable knot.
\end{lem}

\begin{proof}
By Theorem~\ref{thm:essential} the map $h:M \rightarrow E(K)$ is essential.  Let $\tau: A \rightarrow M$ be a double-covering map.  As $\tau_*$ is injective on both $\pi_1(A)$ and $\pi_1(A, \partial A)$, the map $h \circ \tau : A \rightarrow E(K)$ will also be essential.  Note that $(h \circ \tau)|_{\partial A}$ will not be an embedding, but by pushing the image of the two sheets of the covering map off of $h (\partial M)$ we obtain an essential map which is an embedding along the boundary $\partial M$.  Then by Theorem~\ref{thm:cannonandfeustel} there is an essential embedding $f:\Sigma \rightarrow E(K)$, where $\Sigma$ is either a M\"{o}bius band or annulus, and $f(\partial \Sigma)$ is contained in a pair of parallel push-offs of $h(\partial M)$ along $\partial E(K)$.

Suppose first that $\Sigma=M$ is a M\"{o}bius band.  Then $f:M \rightarrow E(K)$ is an embedding, with $f(\partial M)$ a longitude of $\partial E(K)$.  Hence $\gamma_3(K)=1$, and by \cite{clark1978crosscaps} it follows that $K$ is a $(2,q)$-torus or $(2,q)$-cable knot.

One the other hand, if $\Sigma=A$ is an annulus, then we obtain an essential annulus embedded in $E(K)$.  By \cite{simon1973algebraic}, the only such annuli are either subsurfaces of decomposing spheres for $K$, or cabling annuli (see also \cite{ozawa2016knots}).  Since the boundary $\partial f(A)$ is a pair of longitudes of $\partial E(K)$ and not meridians, it follows that $f(A)$ is cannot be extended to a decomposing sphere.  Hence $K$ is a torus or cable knot, with $f(A)$ its cabling annulus.
\end{proof}

Notice that in the above proof the cabling annulus $f(A)$ we obtain in $E(K)$ has the same boundary slope as the original immersed M\"{o}bius band $h(M)$.  In other words, $f(\partial A)$ consists of a pair of simple closed curves in $\partial E(K)$ which are parallel to the longitudinal curve $h(\partial M)$.  

As torus and cable knots are necessarily prime, we have the following immediate corollary.

\begin{cor}
\label{cor:prime}
Any knot $K$ with $\gamma_I(K) = 1$ is prime.
\end{cor}

Now that we've established that any knot $K$ with $\gamma_I(K)=1$ is a nontrivial $(p,q)$-torus or $(p,q)$-cable knot, we proceed to determine what restrictions (if any) this situation places on $p$ and $q$.  We begin by answering the question in the case of torus knots.

\begin{lem}
\label{lem:torusknot}
If $K$ has $\gamma_I(K)=1$ and is a $(p,q)$-torus knot, then one of $p$ or $q$ must be even.
\end{lem}

\begin{proof}
Suppose to the contrary that both $p$ and $q$ are odd.  Suppose further that $K$ sits on the standardly embedded torus $T \subset S^3$ obtained by extending the cabling annulus from the proof of Lemma~\ref{lem:cableortorus}, and let $V$ be one of the solid tori bounded by $T$.  Recall that $\pi_1(E(K))$ can be presented as 
\[
\pi_1(E(K)) = \langle x, y \, | \, x^p = y^q \rangle,
\]
where $x$ and $y$ are the homotopy classes of the cores of the two solid tori in $S^3 \backslash T$.  Suppose, without loss of generality, that $x$ is the homotopy class of the core of $V$.

Note that $T \cap E(K)$ is the cabling annulus obtained in the proof of Lemma~\ref{lem:cableortorus}.  Furthermore, we can homotope $h(\partial M)$ so that it sits on $\partial E(K) \cap \operatorname{int} V$.  Pushing $h (\partial M)$ off of $\partial E(K)$ towards the core of $V$, we see that it is homotopic to $x^p$.

Pushing $\partial M$ inside $M$ towards its core $c$ and tracking its image under $h$, we see that $h(\partial M)$ is also freely homotopic to $h(c)^2$, and hence $h(c)^2$ is conjugate to $x^p$ in $\pi_1(E(K))$.  Note, however, that the only relation the in the above group presentation does not change the parity of the algebraic length of any of the words in $\pi_1(E(K))$.  Hence any representation of $h(c)^2$ as a word in the generators $x$ and $y$ will always have even length, and any word representing a conjugate of $x^p$ will have odd length, a contradiction.  Thus either $p$ or $q$ must be even.
\end{proof}

We now state and prove Proposition~\ref{prop:cableknot}, which will serve to complete the proof of Theorem~\ref{thm:cableknots}.

\begin{prop}
\label{prop:cableknot}
If $K$ has $\gamma_I(K)=1$ and is a $(p,q)$-cable knot, then $p$ must be even.
\end{prop}

We will break the proof of Proposition~\ref{prop:cableknot} into several lemmas.  In what follows, suppose that $K$ is a knot with $\gamma_I(K)=1$, which is a $(p,q)$-cable knot.  Furthermore, suppose that $K$ lies on the boundary $T$ of a solid torus $V$, this time knotted in $S^3$, where $T \cap E(K)$ is the cabling annulus $A$ obtained from the proof of Lemma~\ref{lem:cableortorus}.  Let $W$ be the closure of $S^3 \backslash V$.  Again, $h(\partial M)$ is parallel to the components of $\partial A$ in $\partial E(K)$, but this time we homotope $h(\partial M)$ so that it is embedded and lies on the outside of $V$, in $\partial E(K) \cap W$.

\begin{lem}
\label{lem:outside}
Suppose that $h(M)$ can be homotoped so that it lies entirely in $E(K) \cap W$.  Then $p$ must be even.
\end{lem}

\begin{proof}
Note that $E(K) \cap W$ is homeomorphic to $E(V)=S^3 \backslash V$, which can be viewed as the exterior of the knotted core of $V$.  Then $h(\partial M)$ represents the class $q\cdot m+p\cdot l$ in $\pi_1(\partial E(V))$, where as usual $m$ and $l$ are the homotopy classes in $\pi_1(\partial E(V))$ of the meridian and Seifert-framed longitude respectively.  Note that since $K$ is a cable (and hence a satellite knot) we must have $|p| \geq 2$.  If $p$ is even, then we are done.  Assume therefore, that $p$ is odd.  

By Theorem~\ref{thm:essential} the map $h:M \rightarrow E(V)$ is essential, and hence we can find an essential embedding $f:\Sigma \rightarrow E(V)$, where $\Sigma$ is either an annulus or M\"{o}bius band, and where $f(\partial \Sigma)$ is parallel to $h(\partial M)$ in both cases.  Furthermore, $f(\partial \Sigma)$ will represent either $(q\cdot m + p \cdot l)$ or $(2q \cdot m + 2p \cdot l)$ in $\pi_1 (\partial E(V))$, depending on whether $\Sigma$ is a M\"{o}bius band or annulus respectively.

Suppose first that $\Sigma=M$.  Then the boundary of a tubular neighborhood of $N(f(M))$ will be an embedded, essential annulus $A' \subset E(V)$, whose boundary represents $(2q \cdot m + 2p \cdot l)$ in $\pi_1 (\partial E(V))$.  Then $A'$ must be a cabling annulus for the core of $V$, which implies that $\partial A'$ represents a class of the form $(k\cdot m \pm 2 \cdot l) \in \pi_1 (\partial E(V))$.  Thus $p=\pm 1$, a contradiction.

Suppose then that $\Sigma$ is an annulus.  As above we can conclude that $f(\Sigma)$ is a cabling annulus for the core of $V$, and hence we arrive at the same contradictory conclusion, namely that $p = \pm 1$.
\end{proof}

We now turn our attention to the case when $h(M)$ cannot be arranged to lie entirely in $E(K) \cap W$.  Then $h(M)$ must intersect the cabling annulus $A \subset E(K)$ nontrivially.  We assume that $h$ is transverse to $A$, and hence $h^{-1}(A)$ will be a collection of embedded simple closed curve contained in the interior of $M$.  We thus divert our attention momentarily to discuss such curves on the M\"{o}bius band $M$.

Recall that $M$ is given by the square $[-1,1] \times [-1,1]$ in $\mathbb{R}^2$, with vertical edges identified via $(-1,t) \sim (1,-t)$.  The core of $M$ will is denoted by $c$, and is the image of $[-1,1]\times \{ 0\}$ under the quotient map.  Furthermore, let $\mu$ be the image under the quotient map of the segments $[-1,1] \times \left\{ -\tfrac{1}{2}, \tfrac{1}{2}\right\}$.  While the following result is certainly well-known, we know of no reference to it in the literature, and hence we reproduce its proof here.

\begin{lem}
\label{lem:curvesonM}
Any simple closed curve in $M$ which does not bound a disk is isotopic to either $c$ or $\mu$.
\end{lem}

\begin{proof}
Let $\alpha \subset M$ be a simple closed curve, which we fix an orientation on.  Let $\delta \subset M$ be the image of the arc $\left\{ -1 \right\} \times [-1,1]$ under the quotient map.  Assume that $\alpha$ and $\delta$ intersect transversely, and orient $\delta$ so that the algebraic intersection $\alpha \cdot \delta$ between $\alpha$ and $\delta$ is nonnegative.  Suppose that $|\alpha \cap \delta | > \alpha \cdot \delta$.  Then we can chose an arc $\tau \subset \alpha$ such that the endpoints of $\tau$ consist of both a positive and a negative intersection point of $\alpha$ and $\delta$.  Furthermore, we can assume that there are no other intersection points with $\delta$ on the interior of $\tau$.  Then there is a subarc $\tau'$ in $\delta$, such that $\tau \cup \tau'$ bounds a disk.  After choosing the innermost such disk, we can push $\tau$ through $\delta$ to the other side, removing one pair of cancelling intersection points.  We can repeat this until all remaining intersection points between $\alpha$ and $\delta$ are positive.  If $\alpha \cap \delta = \emptyset$, then $\alpha$ lies in a disk and hence is nullhomotopic.  Assume then that $\alpha \cap \delta \neq \emptyset$.

Lift the simple closed curve $\alpha$ to $[-1,1] \times [-1,1]$, where we get a collection of $n$ properly embedded disjoint arcs $\alpha_1, \ldots, \alpha_n$, each of which has one endpoint on $\left\{-1\right\} \times [-1,1]$ and the other on $\left\{1\right\} \times [-1,1]$.  Assume that the arcs are labelled in order from top to bottom.  The identification of the vertical boundary components then induces an identification of the strands, sending the left endpoint of the $j^\text{th}$ strand to the right endpoint of the $(n-j+1)^\text{th}$ strand.  Represent this identification as an element $\sigma$ of the symmetric group $S_n$ on $n$ letters.  Notice that as a permutation $\sigma \circ \sigma = \text{id}$, however the subgroup of $S_n$ generated by $\sigma$ must act transitively on the set $\left\{ 1, \ldots, n\right\}$ as $\alpha$ is connected.  Thus $n=1$ or $2$, and hence $\alpha$ is isotopic to either $c$ or $\mu$.
\end{proof}

Returning to our map $h : M \rightarrow E(K)$, we show that all loops in $h^{-1}(A)\subset M$ can be avoided except possibly for  curves that are isotopic to $\mu$.  

\begin{lem}
\label{lem:inessentialloops}
The map $h:M\rightarrow E(K)$ can be modified away from $\partial M$ so that all curves in $h^{-1}(A)$ are isotopic to the curve $\mu$ in $M$.
\end{lem}

\begin{proof}
We first note that none of the simple closed curves in $h^{-1}(A)$ can be isotopic to the core $c$ of $M$, since $A$ is orientable and $c$ is an orientation-reversing curve in $M$.

Next we show that we can modify $h$ away from $\partial M$ so that $h^{-1}(A)$ contains no inessential curves.  Let $\alpha \subset h^{-1}(A)$ be a simple closed curve which bounds a disk $D$ in $M$, and assume that $D$ contains no other such curves.  Then $h(D)$ will be an immersed disk which lies in the closure of one of the two components of $E(K) \backslash A$, which we denote by $U$, and whose boundary $h(\alpha)$ is an immersed loop on $A$.  Since $A$ is essential, the immersed loop $h(\alpha)$ will be null-homotopic on $A$.

Pick a homotopy which takes $h(\alpha)$ to a small disk $D'$ in $A$, and extend it to a homotopy of $h$ supported in small neighborhoods of $\alpha$ and $A$, so that the double point curve along $h(\alpha)$ now lies in $D' \subset A$.

Let $N(D')$ be the restriction to $D'$ of a small tubular neighborhood of $A$, which we can parametrize in the usual way as $N(D') = D' \times (-1,1)$.  Then $h(D)$ sits entirely on one side of $D' = D' \times \left\{0\right\}$, so we can assume without loss of generality that $h(D) \cap \left(D' \times (-1,0)\right) = \emptyset$.  

Then $h(D)$ can be thought of as a properly immersed disk in the ball $B = S^3 \backslash (\operatorname{int} D' \times (-1,0))$.  Meanwhile, the surface $h(M\backslash \operatorname{int} D)$ will have one boundary component immersed along $h(\alpha)$, which can be pushed slightly off of $A$ into the interior of $E(K) \backslash U$.  The disk $h(D) \subset B$ can be reglued to the newly repositioned boundary of $h(M\backslash \operatorname{int} D)$, and by shrinking the ball $B$ down sufficiently small we can assume that it is contained entirely in the interior of $E(K) \backslash U$.  The resulting immersion will have one less inessential loop intersection with the annulus $A$.  By removing all such inessential loop intersections, we are left with only with loops in $h^{-1}(A)$ that are isotopic to $\mu \subset M$.
\end{proof}

We thus can assume that $h^{-1}(A)$ consists only of a finite collection of parallel curves $\mu_0, \ldots , \mu_k$, all of which are isotopic to $\mu \subset M$.  Suppose that $\mu_0$ is the innermost of the curves in $h^{-1}(A)$. 

\begin{lem}
\label{lem:boundaryparallelintersections}
The curve $h(\mu_0)$ is homotopic to the core $\kappa$ of the cabling annulus $A$.
\end{lem}

\begin{proof}
We first note that since $\mu_0$ is nontrivial in $\pi_1(M)$ and $h$ is essential, $h(\mu_0)$ will be nontrivial in $\pi_1(A)$.  Thus $h(\mu_0)$ is homotopic to some nonzero power of $\kappa$, say $\kappa^b$, with $b\neq 0$.

Notice now that $h(\partial M)$ is homotopic in $\partial E(K)$ to either of the components of $\partial A \subset \partial E(K)$, both of which are in turn homotopic in $A$ to $\kappa$.  Hence $h(\partial M)$ is homotopic to $\kappa$ in $\pi_1(E(K))$.

On the other hand, $\partial M$ and $\mu_0$ bound an annulus $A_0$ in $M$, and hence $h(\partial M)$ and $h(\mu_0)$ are freely homotopic in $E(K)$ via $h(A_0)$.  This implies that $\left[ h(\partial M) \right]=\left[ h(\mu_0) \right]$ in $H_1(E(K)) \cong \mathbb{Z}$, and hence that $\left[\kappa\right] = b\left[ \kappa \right]$ in homology.  Because the slopes of $\partial A$ and $h(\partial M)$ in $\partial E(K)$ agree, we can compute the boundary slope of $h(\partial M)$ with respect to the Seifert-framed longitude of $K$, to see that its framing coefficient is $pq \neq 0$.  Hence we see that $[h(\partial M)]=[\kappa]$ is nonzero in $H_1(E(K))$.  Thus $b=1$, which completes the proof.
\end{proof}

\begin{proof}[Proof of Proposition~\ref{prop:cableknot}]
Let $M_0$ denote the subsurface of $M$ bounded by $\mu_0$, which will also be a M\"{o}bius band.  Let $h_0:M_0 \rightarrow E(K)$ denote the restriction of $h$ to $M_0$.  Note that $h_0 (\partial M_0)$ will be an immersed curve in the cabling annulus $A$ which is homotopic to the core $\kappa$.  Furthermore, as $\mu_0$ was the innermost curve in $h^{-1}(A)$, $h_0(M_0)$ will be contained entirely inside either $E(K) \cap V$ or $E(K) \cap W$.

Choose a homotopy of $h_0$ which is supported in a small neighborhood of $\partial M_0$, and which first straightens out $h_0(\partial M_0)$ to the embedded core $\kappa$, and then pushes it along $A$ towards one of its boundary components, and finally onto $\partial E(K)$.  If $h_0(M_0) \subset E(K) \cap V$, then we push $h_0(\partial M_0)$ onto $\partial E(K) \cap \operatorname{int} V$, while if $h_0(M_0) \subset E(K) \cap W$ then we push $h_0(\partial M_0)$ onto $\partial E(K) \cap \operatorname{int} W$.

In the latter case, we obtain a proper map $h_0:M_0 \rightarrow E(K)$ whose image is contained entirely outside of $E(K) \cap V$.  Moreover, $h_0$ is essential by Theorem~\ref{thm:essential}, and hence by Lemma~\ref{lem:outside} it follows that $p$ must be even.

Suppose then that $h_0 (M_0) \subset E(K) \cap V$.  Take the solid torus $E(K) \cap V \cong V$, and perform an inverse satellite operation, embedding it in $S^3$ as the standardly embedded solid torus $V'$.  In doing so we chose this embedding so that the longitude coming from the Seifert framing on $V$ is identified with the longitude from the Seifert framing on $V'$, though this will not be necessary.  Using this choice of identification we obtain a map $h':M_0 \rightarrow V'$, with $h'(\partial M_0)$ embedded on $\partial V'$ as a $(p,q)$-torus knot $K'$.  Thinking instead of $h'$ as a map to $E(K')$, it follows immediately that $h'$ is essential.

Now, suppose that $p$ is odd. If we cut the torus $V'$ and apply $k$ twists before regluing, we obtain the solid torus $V'$ again, while the knot $K'$ is transformed into a $(p,q+kp)$-torus knot $K''$ lying on $\partial V'$.  As $p$ is odd, we can choose $k$ so that $q+kp$ is also odd.  However, since the the image of $h'$ is contained entirely in $V'$, we can twist the map $h'$ as well to obtain a new essential map $h'':M_0 \rightarrow E(K'')$, with $h''(\partial M_0)$ a longitude on $\partial E(K'')$.  By Lemma~\ref{lem:torusknot} then either $p$ or $q+kp$ must be even, which is a contradiction.  Hence $p$ must be even.
\end{proof}

\section{Representing $\mathbb{Z}_2$-homology classes with immersed surfaces}
\label{sec:homology}

\begin{proof}[Proof of Theorem~\ref{thm:homology}]
Fix $n \geq 2$, and let $K = T(2n,2n-1)$.  We think of $K$ as being embedded on the boundary of the standard solid torus $V' \subset S^3$ so that the meridian disks of $V'$ intersect $K$ in $2n$ points.  Let $V = V' \cap E(K)$, and $W = E(K) \backslash \operatorname{int} V$, with $A = W \cap V$.  Thus $A$ is the essential cabling annulus for $K$ in $E(K)$, and separates $E(K)$ into two solid torus components $V$ and $W$.   Let $F$ denote an immersed M\"{o}bius band in $V$ constructed as in Proposition~\ref{prop:crosscap1}.

Define $Y_n$ to be the result of Dehn surgery to $S^3$ along $K$ with surgery slope determined by $\partial A \subset E(K)$.  Identifying $E(K)$ as a subspace of $Y_n$ allows us to think of $V,W,A,$ and $F$ as subsets of $Y_n$.  Let $U$ denote the solid torus that is glued to $\partial E(K)$ during the Dehn surgery.

As $\partial F \subset \partial E(K)$ is parallel to the surgery slope, $\partial F$ bounds an embedded disk $D$ in $U$.  Then $F \cup D$ is an immersed $\mathbb{RP}^2$ in $Y_n$, and we denote its $\mathbb{Z}_2$--homology class in $H_2(Y_n; \mathbb{Z}_2)$ by $\alpha_n$.

\begin{figure}
\includegraphics[width=0.4\textwidth]{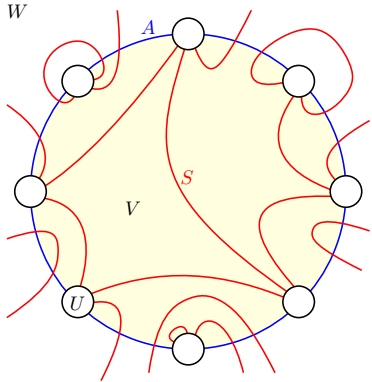}
\caption{A meridian disk of the torus $V$ with an embedded surface $S \subset Y_n$ representing $\alpha_n$.}
\label{fig:RPcolor}
\end{figure}

Suppose now that $\bar{S} \subset Y_n$ is an embedded surface with $\left[\bar{S}\right] = \alpha_n$ in $H_2 (Y_n; \mathbb{Z}_2)$.  After isotopy of $\bar{S}$ if necessary, we can assume that $\bar{S}$ intersects $U$ in a collection of meridian disks $D_1, \ldots, D_m$.  The boundary of these meridian disks will form a collection of longitudes on the torus $\partial E(K) = \partial U$, each of which will be parallel to the surgery slope $\partial A$.  After further isotopy of $\bar{S}$ we can assume that each of these longitudes lies on $\partial E(K) \cap V$ as in Figure~\ref{fig:RPcolor}.  Note that because $\bar{S}$ is homologous to $F \cup D$, which intersects the core of $U$ in a single point, the number $m$ of disks $D_j$ must be odd. Let $S$ be a component of $\bar{S}$ which also intersects $\partial E(K)$ in an odd number of curves.  

We now remove any trivial intersections of $S$ with $A$.  Let $c$ be a component of $A \cap S$ which bounds a disk in $A$, and suppose that $c$ is innermost.  Then we can cut the surface $S$ along $c$ and cap it off with parallel disks on either side of $A$ to remove the intersection loop $c$.  We repeat this procedure for all such intersection loops, always choosing the inner most loops to cut along. If at any stage in this procedure $S$ becomes disconnected, we may discard the piece whose intersection with $\partial E(K)$ has an even number of components.  As these modifications do not decrease the Euler characteristic of $S$, we can thus assume that $S$ is connected and that the components of $A \cap S$ are embedded curves in $A$ that are all parallel to the core of $A$.

\begin{lem}
\label{lem:oddintersections}
The surface $S$ intersects the annulus $A$ in an even number of components.
\end{lem}

\begin{proof}
Suppose to the contrary that there are an odd number of components $l_1, \ldots , l_k$ of $A \cap S$.  Let $S_W = S \cap W$.  Then $S_W$ is a properly embedded surface in $W$, with an odd number of boundary components which are parallel on $\partial W$.  By pushing all  but one of the loops $l_j$ into the interior of $W$ in neighboring pairs  and gluing them together, we obtain a properly embedded surface $S_0 \subset W$ with a single remaining boundary component $l$ on $\partial W$ that is parallel to the core of $A$.  Ignoring any closed components we can assume that $S_0$ is connected.

Choose an embedding of $f: W \rightarrow S^3$ so that $f(W)$ is the standard torus in $S^3$, and so that $f(l)$ is the $(2n-1, 2n)$-torus knot sitting on $\partial f(W)$.  Then $f(S_0)$ will either be a Seifert surface or nonorientable spanning surface for $f(l)$, depending on whether it is orientable or not, and will be contained entirely in $f(W)$.

For any even integer $p\geq 0$, let $f_p : W \rightarrow S^3$ be the embedding obtained by cutting $f(W)$ along a meridian disk, and applying $p$ full twists before regluing.  Then $f_p(l)$ will be a $(2n-1,2n+p(2n-1))$-torus knot, with $f_p(S_0)$ a Seifert or nonorientable spanning surface for $f_p(l)$.

The Seifert genus of the $(2n-1,2n+p(2n-1))$-torus knot is given by
\[
g_3(f_p(l)) = (n-1)(2n-1)(1+p),
\]
while by \cite{teragaito2004crosscap} its crosscap number is 
\[
\gamma_3 (f_p(l)) = \tfrac{1}{2}(p+2n).
\]
Note however that we have a Seifert or nonorientable spanning surface $f_p(S_0)$ for $f_p(l)$ with $\chi (f_p(S_0)) = \chi (S_0)$.  Choosing $p$ sufficiently large then yields a contradiction.
\end{proof}

Consider now the surface $S_V = S \cap V$.  By Lemma~\ref{lem:oddintersections} the surface $S_V$ will have an even number of boundary components on $A$, an odd number on $\partial E(K)$, and hence an odd number in total.  Let $S_1 \subset V$ be a connected component of $S_V$ which also has an odd number of boundary components, which we denote by $c_1, \ldots, c_d$.  Observe that the components of $S \backslash S_1$ will all be surfaces with boundary, none of which can be disks.  Indeed, when thought of a surface in $S^3$, each boundary component of $S \backslash S_1$ is a parallel copy of the non-trivial knot $K = T(2n,2n-1)$.  Thus $\chi(S \backslash S_1)\leq 0$, and we necessarily have $\chi(S_1) \geq \chi(S)$.

As in the proof of Lemma~\ref{lem:oddintersections}, we may push all but one of the parallel boundary components $c_j$ into $V$, where neighboring strands can be paired off and glued together to give a properly embedded surface $S_2 \subset V$ with only a single boundary component.

\begin{lem}
\label{lem:orientablesubsurface}
The surface $S_2$ is nonorientable.
\end{lem}

\begin{proof}
Suppose to the contrary that $S_2$ is orientable.  Then as in the proof of Lemma~\ref{lem:oddintersections} we can choose an embedding $h_p : V \rightarrow S^3$ so that the image of $V$ is the solid torus with $p$ additional twists given along a meridian disk.  Then $h_p(\partial S_2)$ is a $(2n,2n-1+2pn)$-torus knot, which has a Seifert surface $h_p(S_2)$, but has Seifert genus
\[
g_3(h_p(\partial(S_2)) = (2n-1)(n-1+pn).
\]
Taking $p$ sufficiently large again yields the desired contradiction.
\end{proof}

As the nonorientable surface $S_2$ was constructed by identifying neighboring boundary components of the connected surface $S_1$,  if $d \geq 3$ we can find a single pair of neighboring boundary components $c_j$ and $c_{j+1}$ of $S_1$ which, when glued together, yields a nonorientable surface $S_1' \subset V$ from the (possibly orientable) surface $S_1$.  When $d=1$ we simply set $S_1' = S_1=S_2$. In both cases we have that $\chi(S_1') = \chi(S_1) \geq \chi(S)$.

Now take a second copy of the solid torus $V'' = S^1 \times D^2$, and glue its boundary to $\partial V$, so that a meridian of $V''$ is sent to one of the boundary components $c_k$ of $S_1'$ in $\partial V$.  The resulting 3-manifold is the lens space $L(2n,2n-1)$.  Each of the components of $\partial S_1'$ bound a disk inside $L(2n,2n-1)$, and hence we can cap each of them off with disjoint disks to obtain a closed surface $S' \subset L(2n,2n-1)$.  If $d=1$ we have $\chi (S') = \chi (S_1') + 1 \geq \chi(S)+1$, while when $d \geq 3$ we have 
\[
\chi (S') = \chi (S_1') + d-2 \geq \chi(S) +1.
\]
Note however, that Bredon and Wood \cite{bredon1969non} showed that the maximal Euler characteristic of a closed, connected, nonorientable surface embedded in $L(2n,2n-1)$ is $2-n$.  Hence in both cases we have
\[
2-n \geq \chi(S') \geq \chi(S) +1,
\]
from which the theorem follows.
\end{proof}

\bibliographystyle{plain}
\bibliography{bibliography}

\end{document}